\documentclass[twoside,12pt,leqno]{article}
\usepackage{fdeart}
\usepackage{amssymb,amsmath}
\usepackage{graphicx}
\title{
Global attractivity in almost periodic single species models
     }%title

\author
{Y.M.Myslo
\and
V.I.Tkachenko
\thanks {Institute of Mathematics National Academy of Sciences,
Kiev, Ukraine}
}

\def\ReleaseYear{2011}
\def\ReleaseNo{3--4}
\def\ReleaseVol{\,18}

\begin{document}

\maketitle

\begin{abstract}
Using properties of asymptotically almost periodic solutions we prove existence theorem for piece-wise continuous almost
periodic solutions of differential equations with delay and impulses. We apply these results to study almost periodic single species
model with stage structure and impulses.
\end{abstract}
\vspace{3mm}

\begin{keywords}
Almost periodic, delay, impulsive action, single species models, stage structure
\end{keywords}
\vspace{3mm}

\begin{amsmos}
34K45, 34K14, 92D25
\end{amsmos}
\vspace{3mm}

\par

\par

\section{Introduction}
We investigate the existence and attractivity properties of piece-wise continuous almost periodic
solutions for systems of differential equations with delay and impulsive action. Since solutions of impulsive system
have discontinuities, almost periodicity of impulsive system can be understood in different way.
In our paper we use conception of discontinuous almost periodic function proposed in \cite{HW} and
then investigated in \cite{PA,SP,SPA,ST1,ST3} and other works.
Following ideas of \cite{Y}, we first prove the existence theorem for an asymptotically almost periodic solution of impulsive system.
This implies the existence of an discontinuous almost periodic solution.

Single species model with time-delay stage structure was proposed in \cite{AF} and then many authors studied
different kinds of stage structure biological models (see, for example, \cite{CCA,Ch,MT,MT2,X}).
Using results of first part of our paper we obtain sufficient conditions for the existence and global attractivity
of discontinuous almost periodic solutions of almost periodic single species models with stage structure and impulses.

\section{Discontinuous almost periodic solutions}
We will consider the space $\mathcal{PC}^k(J,\mathbb{R}^n), \ J \subset \mathbb{R},$
of all piece-wise continuous functions $x: J \to \mathbb{R}^n$ such that

i) the set $T = \{ t_j \in J, t_{j+1} > t_j, j \in \mathbb{Z}\}$ is the set of discontinuities of $x$;

ii) $x(t_j - 0) = x(t_j)$ and there exists $\lim_{t \to t_j +0} x(t) = x(t_j + 0) < \infty;$

iii) the function $x(t)$ is $C^k$ smooth in $J \setminus T.$

\begin{definition} \cite{SP}
The sequence $\{t_k\}$ has uniformly almost periodic differences if for any $\varepsilon > 0$
there exists a relatively dense set of $\varepsilon$-almost periods common for all sequences
$\{t^j_k\},$ where $t^j_k = t_{k+j} - t_{k}, j \in \mathbb{Z}.$
\end{definition}

\begin{definition} \cite{SP}
The function $\varphi(t) \in \mathcal{PC}^k(\mathbb{R},\mathbb{R}^n)$ is said to be W-almost periodic (W.a.p.) if

i) the sequence $\{t_k\}$ of discontinuities of $\varphi(t)$ has uniformly almost periodic differences;

ii) for any $\varepsilon > 0$ there exists a positive number $\delta = \delta(\varepsilon)$ such that if the points
$t'$ and $t''$ belong to the same interval of continuity and $|t' - t''| < \delta$ then $\|\varphi(t') - \varphi(t'')\| < \varepsilon$
($\|.\|$ is usual norm in $\mathbb{R}^n$);

iii) for any $\varepsilon > 0$ there exists a relatively dense set $\Gamma$ of $\varepsilon$-almost periods such that if
$\tau \in \Gamma,$ then $\|\varphi(t +\tau) - \varphi(t)\| < \varepsilon$ for all $t \in \mathbb{R}$ which satisfy
the condition $|t - t_k| \ge \varepsilon, k \in \mathbb{Z}.$
\end{definition}

\begin{definition} \cite{PA}
Piece-wise continuous function $\varphi_1(t) \in \mathcal{PC}(J,\mathbb{R}^n)$ is situated in the $\varepsilon$-neighborhood of
function $\varphi_2(t) \in \mathcal{PC}(J,\mathbb{R}^n)$ if $\|\varphi_1(t) - \varphi_2(t)\| < \varepsilon$ for all $t \in J$ such that
$|t - \tau^1_i| > \varepsilon, |t - \tau^2_i| > \varepsilon,$ and $|\tau^1_i - \tau^2_i| < \varepsilon, i \in \mathbb{Z},$
where $\{\tau^1_i\}$ and $\{\tau^2_i\}$ are sequences of discontinuities of $\varphi_1(t)$ and $\varphi_2(t)$ respectively.
In this case we will write $\rho(\varphi_1,\varphi_2) < \varepsilon.$

The sequence $\{f_k(t)\}$ of functions $f_k \in \mathcal{PC}(J,\mathbb{R}^n), J \subset \mathbb{R},$ converges in W-topology to function
$f \in \mathcal{PC}(J,\mathbb{R}^n)$ if for any $\varepsilon > 0$ there exists positive integer $N = N(\varepsilon)$ such that
$\|f_k(t) - f(t)\| < \varepsilon$ for all $k \ge N$ and $|t - \tau_i| > \varepsilon$ ($\tau_i$ are points of discontinuities of the function $f$
at the set $J$) and points of discontinuities of functions $f_k(t)$ which are contained in $J$ converges to points $\tau_i$ uniformly with
respect to $i$.
\end{definition}

We consider the system with delay and impulsive action
\begin{eqnarray}
& & \frac{dx(t)}{dt} = f(t,x(t),x(t-h)), \ t \not= t_k, \label{1} \\
& & x(t_k+0) = x(t_k) + I_k(x(t_k)), \ k \in \mathbb{Z}, \label{2}
\end{eqnarray}
where $x \in \mathbb{R}^n, h = const > 0.$ We assume that

1) sequence of real numbers $t_k$ has uniformly almost periodic differences;

2) function $f(t,x,y)$ is W-almost periodic in $t$ and Lipschitz in $x$ and $y$ uniformly for $x, y$ from compact sets,
the sequence of discontinuities of $f$ is the sequence $\{t_k\};$

3) the sequence $\{I_k(x)\}$ is almost periodic uniformly with respect to $x$ from compact sets. Functions $I_k(x)$ are Lipschitz in $x.$

We denote by $x_t$ the function $x(t+\theta), \theta \in [-h,0],$ where $x(t) \in \mathcal{PC}(\mathbb{R},\mathbb{R}^n).$

\begin{definition}
The piece-wise continuous function $a(t)$ is W-asymptotically almost periodic (W.a.a.p.) if it is a sum of W-almost periodic
function $p(t)$ and function $q(t) \in \mathcal{PC}$ such that $q(t) \to 0$ as $t \to \infty.$
\end{definition}

\begin{proposition}
A solution $\xi(t)$ of system (\ref{1}), (\ref{2}) is W.a.a.p. if and only if for any sequence
of real numbers $\{\tau_k\}$ such that $\tau_k \to \infty$
as $k \to \infty$ there exists a subsequence $\{\tau_{k_j}\}$ for which $\xi(t + \tau_{k_j})$ converges on
$0 \le t < \infty$ in W-topology.
\end{proposition}

\begin{proof} {\it Necessity.} Let $\xi(t)$ be W.a.a.p., then
$\xi(t) = p(t) + q(t)$ where $p(t)$ is W.a.p. and $q(t) \to 0$ as $t \to \infty.$
By \cite{PA}, a piece-wise continuous function $p(t)$ is W.a.p. if and only if every infinite set of shifts
$\{\varphi(t + \tau_n)\}$ is compact relative to W-topology. Since $q(t + \tau_k) \to 0$ as $\tau_k \to \infty$  for all $t\ge 0,$
then there exists a subsequence $\{\tau_{k_j}\}$ for which $\xi(t + \tau_{k_j})$ converges on
$0 \le t < \infty$ in W-topology.

{\it Sufficiently.} Since the sequence $\{t_k\}$ of points of impulses
has uniformly almost periodic differences then by \cite{ST1, ST3}
for any sequence $\{\tau_k\}$ there exist subsequence (which we denote by $\{\tau_k\}$ again),
 sequence $\{p_k\}$ with uniformly almost periodic differences,
and sequence $\{\alpha(k)\}$ such that
\begin{eqnarray} \label{li}
\lim_{k \to \infty}(t_{n + \alpha(k)} - \tau_k) = p_n
\end{eqnarray}
uniformly in $n \in \mathbb{Z}.$
Taking into account (\ref{li}) and the W-convergence of $\xi(t + \tau_{k})$, we have that for every $\varepsilon > 0$ there
exists positive integer $N = N(\varepsilon)$ such that
$|t_{n + \alpha(k)} - \tau_k - p_n| < \varepsilon$ for all $k \ge N$ and $n \in \mathbb{Z}$
and $\|\xi(t + \tau_k) - \xi(t + \tau_m)\| < \varepsilon$ for all $k,m \ge N$ and $t \in [0, \infty), |t + \tau_k - t_n| > \varepsilon,
|t + \tau_m - t_n| > \varepsilon.$

First, we prove that for every $\varepsilon > 0$ there is $T = T(\varepsilon)$ such that the set
\begin{eqnarray} \label{set1}
\{ \tau: \ \sup_{t \ge T(\varepsilon)}\|\xi(t+\tau) - \xi(t)\| < \varepsilon, \ |t - t_k| > \varepsilon\}
\end{eqnarray}
is relatively dense in $\mathbb{R}.$ Conversely, suppose that there exists $\varepsilon_0 > 0$ such that set
(\ref{set1}) is not relatively dense. Then for this $\varepsilon_0$ and any $T(\varepsilon_0)$ there exists
a sequence of intervals $[h_n-l_n, h_n+l_n]$ such that
$\sup_{t \ge T(\varepsilon_0), |t-t_k| \ge \varepsilon_0}\|\xi(t+\tau) - \xi(t)\| \ge \varepsilon_0$ for all points $\tau \in [h_n-l_n, h_n+l_n].$
Let $l_1$ be arbitrary and $l_n > \max_{m < n} l_m$, then
$h_n - h_m \in [h_n-l_n, h_n+l_n]$ if $m < n.$
Therefore
$$\sup\limits_{t \ge h_n, |t + h_n-t_k| \ge \varepsilon_0}\|\xi(t + h_n) - \xi(t + h_m)\| =
 \sup\limits_{t \ge 0, |t-t_k| \ge \varepsilon_0}\|\xi(t) - \xi(t + h_n - h_m)\| \ge \varepsilon_0.$$
This contradicts the W-convergence of $\{\xi(t + h_n)\}.$

Assume that the sequence $\{\xi(t+\tau_k)\}$ converges on $[0,\infty)$ in W-topology to function $p(t)$ with the sequence of discontinuities
$\{p_n\}.$ As a limit of shifts of sequence $\{t_n\}$ the sequence $\{p_n\}$ has uniformly almost periodic differences.
Choosing a subsequence of $\{\tau_k\}$ if necessary we can construct function $p(t)$ on the hole axis such that
$\xi(t+\tau_k)$ converges to $p(t)$ in W-topology on compact subsets of $\mathbb{R}.$

Now we prove that $p(t)$ is W.a.p. We have proved that if $\varepsilon > 0$ is given then there exist $T(\varepsilon)$
and relatively dense set of numbers $\tau$ such that $\|\xi(t+\tau) - \xi(t)\| < \varepsilon$ if $t \ge T(\varepsilon), t + \tau \ge T(\varepsilon)$
and $|t - t_k| > \varepsilon.$
Introducing $\tau_n$ with sufficiently large $n$ we have
$\|\xi(t+ \tau_n + \tau) - \xi(t + \tau_n)\| < \varepsilon$
if $t \ge T(\varepsilon) - \tau_n, \ t
+ \tau \ge T(\varepsilon) - \tau_n$
and $|t + \tau_n - t_k| > \varepsilon.$ Fix $t$ and $\tau$ and take $n$ large enough so the last inequalities are correct. Then taking limits
$n \to \infty$ we have $\|p(t+\tau) - p(t)\| < \varepsilon.$ This holds for $t \in \mathbb{R}, |t - p_k| > \varepsilon$
and relatively dense set of
$\varepsilon$-almost periods $\tau.$ Thus $p(t)$ is W.a.p.
Analogously with \cite{F}, p. 158, we can show that $\xi(t) - p^*(t) \to 0$ as $t \to \infty,$ where W.a.p. function $p^*(t)$
is some translation of W.a.p. function $p(t).$
\end{proof}

\begin{theorem}
Suppose that system (\ref{1}), (\ref{2}) has a solution $\xi(t)$ defined on $I = [0, \infty)$ such that
$\|\xi(t)\| \le B$ for all $t \ge 0.$ If $\xi(t)$ is W.a.a.p. then the system (\ref{1}), (\ref{2}) has an W.a.p.
solution $p(t).$
\end{theorem}

\begin{proof}
W.a.a.p. solution $\xi(t)$ has form $\xi(t) = p(t) + q(t)$ where $p(t)$ is W.a.p. and
$q(t) \to 0$  as $t \to \infty.$
Let $\{\tau_k\}$ be a sequence such that there exist a sequence $\{\alpha(k)\}$
and $N$ such that $|t_{n+\alpha(k)} - \tau_k - t_n| < \varepsilon$ and $\rho(p(t + \tau_k), p(t)) < \varepsilon$
for all $n \in \mathbb{Z}$ and $k \ge N.$
Obviously, $q(t + \tau_k) \to 0$ as $k \to \infty.$

The function $\xi(t + \tau_m)$ satisfies equation
\begin{eqnarray}
& & \dot x(t) = f(t + \tau_m, x(t), x(t - h)), \label{1m}\\
& & x(t_n - \tau_m + 0) = x(t_n - \tau_m) + I_n(x(t_n - \tau_m)), \ n = 0,1,... \label{2m}
\end{eqnarray}

Denote $t_n - \tau_m = t'_{n - \alpha(m)}.$ Let $n - \alpha(m) = j,$ then $n = j + \alpha(m)$ and  equation (\ref{1m}), (\ref{2m})
is written in the form
\begin{eqnarray}
& & \dot x(t) = f(t + \tau_m, x(t), x(t - h)), \label{11m}\\
& & x(t'_j + 0) = x(t'_j) + I_{j + \alpha(m)}(x(t'_j)), \ j = 0,1,... \label{22m}
\end{eqnarray}

For $\varepsilon > 0$ there exists $\delta = \delta(\varepsilon) > 0$ such that
\begin{eqnarray*}
& & \| f(t + \tau_m, \xi(t + \tau_m), \xi(t + \tau_m-h)) - \\
& & \qquad \qquad - f(t + \tau_m, p(t + \tau_m), p(t + \tau_m-h))\| < \varepsilon, \\
& & \| I_{j + \alpha(m)}(\xi(t'_j)) - I_{j + \alpha(m)}(p(t'_j)) \| < \varepsilon
\end{eqnarray*}
if $\|\xi(t + \tau_m) - p(t + \tau_m)\| = \| q(t + \tau_m)\| < \delta$ and $\| \xi(t'_j) - p(t'_j)\| = \| q(t'_j)\| < \delta.$

Let $[\bar t_1, \bar t_2]$ be some subinterval of $\mathbb{R}.$
 For $\delta(\varepsilon)$ there exists positive integer $N$ such that
$\rho(p(t),p(t + \tau_m)) < \delta$ for all $[\bar t_1, \bar t_2]$ and $m \ge N.$

We write system (\ref{11m}), (\ref{22m}) in the integral form
\begin{eqnarray*}
& & \xi(t + \tau_m) = \xi(\bar t_1 + \tau_m) + \int_{\bar t_1}^t f(s + \tau_m, \xi(s + \tau_m), \xi(s + \tau_m-h))ds + \\
& & + \sum_{\bar t_1 \le t'_j < t} I_{j + \alpha(m)}(\xi(t'_{j} + \alpha_m).
\end{eqnarray*}
Making $\tau_m \to \infty$ we have
\begin{eqnarray*}
p(t) = p(\bar t_1) + \int_{\bar t_1}^t f(s, p(s), p(s-h))ds + \sum_{\bar t_1 \le t_j < t} I_{j}(p(t_{j}).
\end{eqnarray*}
Hence, W.a.p. function $p(t)$ is a solution of system (\ref{1}), (\ref{2}).
\end{proof}

\begin{theorem}
Let bounded solution $\xi(t)$ of the system (\ref{1}), (\ref{2}) be uniformly asymptotically stable for $t \ge 0.$
Then $\xi(t)$ is W.a.a.p. and system (\ref{1}), (\ref{2}) has W.a.p. solution which is uniformly asymptotically stable
for $t \ge 0.$
\end{theorem}
\begin{proof} Since $\xi(t)$ is uniformly asymptotically stable, then for any $\varepsilon > 0$ there exist
$\delta = \delta(\varepsilon) > 0$ and $T(\varepsilon) > 0$ such that if $x(t)$ is a solution of (\ref{1}), (\ref{2})
such that $\rho(\xi_0, x_0) < \delta$ then
$\| \xi(t) - x(t)\| < \varepsilon/2$ for $t \ge 0$ and $\| \xi(t) - x(t)\| < \delta_1/2$ for all
$t \ge T(\varepsilon), \ \delta_1 = \min(\varepsilon,\delta).$

Let $\{\tau_m\}$ be a sequence such that $\tau_{m+1} > \tau_m, \tau_m \to \infty$ as $m \to \infty.$

Denote $\xi^m(t) = \xi(t + \tau_m).$ Then $\xi^m(t)$ is a solution of system (\ref{11m}), (\ref{22m})
and $\xi^m(t)$ is uniformly asymptotically stable with same $\delta(\varepsilon)$ and $T(\varepsilon)$ as for $\xi(t).$

There exists a subsequence of $\{\tau_m\}$ (which we denote by $\{\tau_m\}$ again) such that there exist
sequence $\{p_n\}$ with uniformly almost periodic differences and sequence $\alpha(m)$ such that:

$\lim_{m \to \infty}(t_{i+\alpha(m)} - \tau_m) = p_i, \ \lim_{m \to \infty}(I_{i+\alpha(m)}(x) = J_i(x)$
uniformly with respect to $i \in \mathbb{Z}$ and $\|x\| \le K,$

$f(t + \tau_m,x,y)$ tends to $g(t,x,y)$ in W-topology uniformly with respect to $x,y, \|x\| \le K, \|y\| \le K,$

$\xi_0^m = \{\xi^m(\theta), \theta \in [-h,0]\}$ converges in W-topology to $\zeta_0 = \{\xi(\theta), \theta \in [-h,0]\}.$
%\vspace{1mm}

Therefore, for some $\delta_2 > 0$ there exists a positive integer $k_0(\varepsilon)$
such that if $k \ge m \ge k_0(\varepsilon)$ then
$$\|f(t+\tau_k,x,y) - f(t+\tau_m,x,y)\| < \delta_2(\varepsilon)$$
for all $\|x\| \le K, \|x\| \le K$ and all $t \in \mathbb{R}, |t + \tau_k - t_j| > \delta_2, |t + \tau_m - t_j| > \delta_2, j \in \mathbb{Z},$
and
$$\|\xi(\theta + \tau_k) - \xi(\theta + \tau_m)\| <  \delta_1(\varepsilon)$$
for all $\theta \in [-h,0], |\theta + \tau_k - t_j| > \delta, |\theta + \tau_m - t_j| > \delta, j \in \mathbb{Z}.$

Let $\eta(t)$ be the solution of system (\ref{11m}), (\ref{22m}) with initial function $\eta_0 = \xi^k_0.$
Since system (\ref{11m}), (\ref{22m}) is uniformly asymptotically stable then
$\| \xi^m(t) - \eta(t)\| < \varepsilon/2$ for all $t \ge 0$ and
$\| \xi^m(t) - \eta(t)\| < \delta_1/2$ for all $t \ge T(\varepsilon).$

Let us estimate the difference $\eta(t) - \xi^k(t)$ if $t \in [0,T(\varepsilon) + h].$
The function $\xi^k(t)$ satisfies equation
\begin{eqnarray}
& & \dot y(t) = f(t + \tau_k, y(t), y(t - h)), \label{111m}\\
& & y(t''_j + 0) = y(t''_j) + I_{j + \alpha(k)}(y(t''_j)), \ j = 0,1,..., \label{222m}
\end{eqnarray}
where $t''_j = t_{j + \alpha(k)} - \tau_k.$

Difference $\eta(t) - \xi^k(t)$ satisfies following integral equation
\begin{eqnarray*}
& & \eta(t) - \xi^k(t) = \eta(0) - \xi^k(0) + \\
& & + \int_0^t \left( f(s + \tau_m, \eta(s),\eta(s-h)) - f(s + \tau_k, \xi^k(s),\xi^k(s-h))\right) ds + \\
& & + \sum_{0 \le t'_j < t}I_{j + \alpha(m)}(\eta(t'_j)) - \sum_{0 \le t''_j < t}I_{j + \alpha(k)}(\xi(t''_j)).
\end{eqnarray*}

By \cite{SP}, p. 191, there exist a number $l > 0$ and a positive integer $q$ such that any interval of the time axis of
length $l$ contains no more then $q$ terms of the sequence $\{p_n\}.$
Since $t'_n = t_{n+\alpha(m)} - \tau_m, t''_n = t_{n+\alpha(k)} - \tau_k,$ and $t'_n \to p_n, t''_n \to p_n,$
as $m \to \infty, k \to \infty,$ then interval $[0, T(\varepsilon)+h)$ contains finite number of points from
the sequence $\{p_n\}$ and the same number of points $t'_n$ and $t''_n$ (if $k$ and $m$ are sufficiently large.)

Analogously to \cite{DB}, Theorem 5, we estimate difference $\eta(t) - \xi^k(t)$ in succession on interval
$[0, \min(t'_1, t''_1)],$ at point $\max(t'_1, t''_1),$ on the interval $[\max(t'_1, t''_1), \min(t'_2, t''_2)]$ and so on.

As result, for $\delta_1/2$ there exists $\delta_2(\varepsilon)$ such that if $|t'_j - t''_j| < \delta_2, \
\|I_{j+q(k)} - I_{j+q(m)}\| < \delta_2, \ \|f(t+\tau_k,x,y) - f(t+\tau_m,x,y)\| < \delta_2$ then
$|\xi^k(t) - \eta(t)| < \delta_1/2$ for all $t \in [0, T(\varepsilon)+h], |t + \tau_k - t_j| > \delta/2, |t + \tau_m - t_j| > \delta/2,j \in \mathbb{Z}.$

Therefore, we have inequality $\rho(\xi^k, \xi^m) < \varepsilon$ on the interval $t \in [0,T(\varepsilon)+h]$ and
$\rho(\xi^k, \xi^m) < \delta_1 \le \delta$ on the interval $t \in [T(\varepsilon),T(\varepsilon)+h].$

By the same argument as the above, we can see that if $m \ge k \ge k_0(\varepsilon),$ then
$$\rho(\xi^k_{t},\xi^m_{t}) < \varepsilon \quad {\rm for} \quad  t \in [T(\varepsilon), 2T(\varepsilon)]$$ and, in general,
$$\rho(\xi^k_{t},\xi^m_{t}) < \varepsilon \quad {\rm for} \quad t \in [qT(\varepsilon), (q+1)T(\varepsilon)], q =1,2,....$$

Hence, $\xi$ is W.a.a.p. and equation has W.a.p. solution.
\end{proof}

\section{Stage structure model}

We consider the system of differential equations with impulsive action, which describe the behavior of
biological species with two stages, immature and mature
\begin{eqnarray}
& & \dot x_i(t) = \alpha(t)x_m(t) -\gamma(t)x_i(t) - \alpha(t-h)e^{-\int_{t-h}^t\gamma(s)ds} x_m(t-h), \label{sst1}\\
& & \dot x_m(t) = \alpha(t-h)e^{-\int_{t-h}^t\gamma(s)ds} x_m(t-h) - \beta(t)x_m^2(t), \label{sst2}
\end{eqnarray}
for $t \not= t_k$ and impulsive action
\begin{eqnarray}
  x_m(t_k + 0) = (1 + d_k) x_m(t_k), \label{sst3}
\end{eqnarray}
at moments $t_k, k \in \mathbb{Z}.$ We assume that
the sequence $\{t_k\}$ of moments of impulsive action has uniformly almost periodic differences,
the sequence $\{d_k\}$ is almost periodic, $d_k \in (-1,d], d>0,$
functions $\alpha(t), \beta(t)$ and $\gamma(t)$ are piece-wise continuous positive and W-almost periodic.
For the sake of simplicity, we assume that points of discontinuities of $f$ are $t_k, k \in \mathbb{Z}.$

Here, $x_i(t)$ and $x_m(t)$ denote the density of immature and mature populations respectively.
The birth of immature population at time $t > 0$ is proportional to the existing mature population with
birth rate $\alpha(t),$  $\gamma(t)$ is the immature death rate, $\beta(t)$ is the mature
death and overcrowding rate, $h$ is the time to maturity. The term $\alpha(t-h)e^{-\int_{t-h}^t\gamma(s)ds} x_m(t-h)$
represents the transformation of immatures to matures.

According to biological interpretation we consider nonnegative solutions of (\ref{sst1}) - (\ref{sst3})
with initial conditions
\begin{eqnarray}
& & x_i(0) = \varphi_i > 0, \label{ini1} \\
& & x_m(\theta) = \psi_m(\theta) \ge 0, \ \theta \in [-h, 0], \ \psi_m(0) > 0. \label{ini2}
\end{eqnarray}

For a function $g(t)$ bounded on the real axis, we denote $g^L = \inf_t g(t), g^M = \sup_t g(t).$

For almost periodic sequence $\{d_n\},$ there exists
$$\sigma = \lim_{T \to \infty} \frac{1}{T}\sum_{0 \le t_k < T}\ln (1 + d_k).$$
We assume that the function
$$\omega (t) = \prod_{0 \le t_k < t}(1 + d_k) e^{-\sigma t}$$ is W-almost periodic.
Then functions
\begin{eqnarray*}
A(t) = \prod_{t-h \le t_k < t}(1+ d_k)^{-1}\alpha(t-h) \exp\left(\sigma h - \int_{t-h}^t \gamma(s)ds\right)
\end{eqnarray*}
and
\begin{eqnarray*}
C(t) = \prod_{0 \le t_k < t}(1+ d_k)e^{-\sigma t} \beta(t)
\end{eqnarray*}
are also W-almost periodic.

\begin{theorem} \label{theorem3}
Assume that the inequality
\begin{eqnarray}
\sigma + \sup_t \left( \alpha(t-h)e^{\sigma -\int_{t-h}^t\gamma(s)ds}\prod_{t-h \le t_k < t}(1 + d_k)^{-1}\right) >0 \label{od44}
\end{eqnarray}
is fulfilled. Then system (\ref{sst1}) - (\ref{sst3}) is permanent, i.e.,
there exist positive constants $m_0$ and $M_0$ such that
all its solutions with
initial values (\ref{ini1}), (\ref{ini2}) satisfy inequalities
$\liminf_{t \to \infty} x_i(t) \ge m_0, \  \limsup_{t \to \infty} x_i(t) \le M_0, \
\liminf_{t \to \infty} x_m(t) \ge m_0, \  \limsup_{t \to \infty} x_m(t) \le M_0.$

If, in addition, the inequality
 \begin{eqnarray}
(A^M + \sigma)C^M < 2 C^L (A^L + \sigma) \label{od444}
\end{eqnarray}
is satisfied, then system (\ref{sst1}) - (\ref{sst3}) has unique positive W-almost periodic solution
which is globally attractive.
\end{theorem}

\begin{proof}
First, we prove that $x_m(t,\varphi) > 0, t > 0$ if
$\psi_m(\theta) \ge 0, \ \theta \in [-h, 0], \ \psi_m(0) > 0.$
Really, if $t \in [0,h]$ then equation (\ref{sst2}), (\ref{sst3}) has form
\begin{eqnarray}
& & \dot x_m(t) = \alpha(t-h)e^{-\int_{t-h}^t \gamma(s)ds} \psi_m(t - h)  - \beta(t) x_m^2(t), \label{od5} \\[2mm]
& & x_m(t_k+0) = (1 + d_k)x_m(t_k). \label{od6}
\end{eqnarray}
The solution of equation (\ref{od5}) - (\ref{od6}) with initial value $x_m(0) = \psi_m(0) > 0$
is estimated from below by corresponding solution of the equation
\begin{eqnarray*}
& & \dot u(t) = - \beta(t)u^2(t), \ t \not= t_k, \\
& & u(t_k+0) = (1 + d_k)u(t_k).
\end{eqnarray*}
The solution of last equation is strictly positive for $t \in (0,h]$ since $u(0) = \varphi(0) > 0$ and $(1 + d_k) > 0.$
Analogously, considering equation on intervals $[h,2h], [2h,3h],...,$ we obtain positiveness for all $t > 0.$

To prove  $x_i(t) > 0, t > 0,$ we use the following argument.
The number of immatures which was born at time $s$ and survived to time $t$
is given by $\alpha(s)x_m(s) e^{-\int_s^t \gamma(\xi)d\xi}.$
Since $t-s \le h,$ then
\begin{eqnarray} \label{x1}
x_i(t) = \int_{t-h}^{t} \alpha(s)x_m(s) e^{-\int_s^t \gamma(\xi)d\xi} ds. %\prod_{s \le t_k < t}(1 + d_{k}) ds.
\end{eqnarray}
By (\ref{x1}) we have  $x_i(t) > 0$ for all $t \ge 0$ since $x_m(s) > 0.$

We make change of variables
\begin{eqnarray} \label{change}
x_m(t) = \omega(t)v(t) = \prod_{0 \le t_k < t}(1+ d_k)e^{-\sigma t}v(t)
\end{eqnarray}
at the equation (\ref{sst2}) - (\ref{sst3}) and obtain the following
equation without impulses
\begin{eqnarray}
\dot v(t) = A(t)v(t-h) + \sigma v(t) -C(t)v^2(t). \label{bez1}
\end{eqnarray}
Solutions of equation (\ref{bez1}) are continuous with left continuous derivatives.

Parallel with (\ref{bez1}) we consider two equations
\begin{eqnarray}
& & \dot v_L(t) = A^L v_L(t-h) + \sigma v_L(t) -C^M v_L^2(t), \label{bez1L} \\
& & \dot v_M(t) = A^M v_M(t-h) + \sigma v_M(t) -C^L v_M^2(t), \label{bez1M}.
\end{eqnarray}
Let $v(t,\varphi), v_L(t,\varphi)$ and $v_M(t,\varphi)$ be solutions of equations
(\ref{bez1}), (\ref{bez1L}) and (\ref{bez1M}) respectively with the same initial function. By \cite{HS}, p. 79, they satisfy
inequalities
\begin{eqnarray}
v_L(t,\varphi) \le v(t,\varphi) \le v_M(t,\varphi), \ t \ge 0.
\end{eqnarray}

By \cite{K}, equation (\ref{bez1L}) has unique positive asymptotically stable equilibrium
$\bar m_0 = (A^L + \sigma)/C^M$ (if $A^L > - \sigma$). Analogously, equation (\ref{bez1M}) has unique positive asymptotically stable equilibrium
$ \bar M_0 = (A^M + \sigma)/C^L.$ Therefore,
\begin{eqnarray} \label{insu}
\liminf_{t \to \infty} v(t,\varphi) \ge \frac{A^L + \sigma}{C^M}, \ \limsup_{t \to \infty} v(t,\varphi) \le \frac{A^M + \sigma}{C^L}.
\end{eqnarray}
The permanence of equation (\ref{sst1}) follows from (\ref{x1}) directly.

Now we prove uniform asymptotic stability of solutions of (\ref{bez1}). We consider two solutions $x(t)$ and $y(t)$
such that $x(t) \ge m_0-\delta$ and $y(t) \ge \bar m_0 - \delta$ for all $t \ge 0,$ where $\delta$ is some small positive constant.
The difference $w(t) = x(t) - y(t)$ satisfies linear equation
\begin{eqnarray} \label{zn3}
\frac{d}{dt}w(t) = A(t)w(t-h) - w(t)\left(C(t)(x(t) + y(t)) -\sigma\right).
\end{eqnarray}
By \cite{HL}, p. 111, equation (\ref{zn3}) is uniformly asymptotically stable if
$$A^M < \inf_t \left(C(t)(x(t) + y(t)) - \sigma \right).$$
Using (\ref{insu}), we obtain (\ref{od444}). By Theorem 2, equation (\ref{sst2}), (\ref{sst3}) has unique
positive piece-wise continuous almost periodic solution which is globally attractive.
By (\ref{x1}), we verify W-almost periodicity of $x_i(t)$ if $x_m(t)$ is W-almost periodic.

\end{proof}

\end{document}